\documentclass[11pt]{amsart}
 \pdfoutput=1
\usepackage{graphicx}

\usepackage{amssymb}
\usepackage{epstopdf}
\usepackage{amsmath}
\usepackage{amscd}
\usepackage{amsthm,amsfonts,amsxtra,amssymb,amscd}
\usepackage{enumerate}

\usepackage{hyperref}

\theoremstyle{plain}
\newtheorem*{lemma*}{Lemma}
\newtheorem{lemma}[subsection]{Lemma}
\newtheorem*{theorem*}{Theorem}
\newtheorem{theorem}[subsection]{Theorem}
\newtheorem*{proposition*}{Proposition}
\newtheorem{proposition}[subsection]{Proposition}
\newtheorem*{corollary*}{Corollary}

\theoremstyle{definition}
\newtheorem*{definition*}{Definition}
\newtheorem{definition}[subsection]{Definition}
\newtheorem*{example*}{Example}
\newtheorem{example}[subsection]{Example}
\theoremstyle{remark}
\newtheorem*{remark*}{Remark}
\newtheorem{remark}[subsection]{Remark}

\title[Vector partition functions]{ Vector partition functions  and generalized Dahmen-Micchelli spaces}
\author{
C. De Concini,\quad
C. Procesi.\quad M. Vergne}
\thanks{The first two authors are partially supported by the Cofin 40 \%, MIUR}

\begin{document}

\begin{abstract}
This is the first of a series of papers  on partition functions and the index theory of transversally elliptic operators.
In this paper we only discuss algebraic and combinatorial issues related to partition functions. The applications to index theory will  appear in   \cite{dpv}.
\smallskip

Here we introduce  a  space of functions on a lattice which generalizes
the space of quasi--polynomials satisfying the difference
equations associated to cocircuits of a sequence of vectors $X$.
This space $  \mathcal F(X)$ contains the partition function ${\mathcal P }_X
$. We prove a "localization formula" for any $f$ in $\mathcal F(X)$.  In
particular, this implies that the partition function ${\mathcal P}_X $ is a quasi--polynomial on the sets
${\mathfrak c}-B(X)$ where ${\mathfrak c}$ is a big cell and $B(X)$ is the {\it zonotope} generated by the vectors  in $X$.
\end{abstract}\maketitle

\section{Introduction}

\smallskip

Recall  some notions. We take  a lattice $\Gamma$ in a vector
space $V$  and  $X:=[a_1,\dots,a_m]$
a list of  non zero elements of
$\Gamma$, spanning $V$ as vector space.
 If    $X$ generates a
pointed cone $C(X)$,  the  {\it partition function}  ${\mathcal P }_X
(\gamma)$ counts the number of ways in which a vector $\gamma\in
\Gamma$ can be written as $\sum_{i=1}^m k_ia_i$ with
$k_i$ non negative integers.

A quasi--polynomial is a function on $\Gamma$ which coincides with
a polynomial on each coset of some sublattice of finite index  in
$\Gamma$. A  theorem \cite{DM},\cite{SV1}, generalizing the theory
of the Ehrhart polynomials \cite{Er1},  shows that ${\mathcal P
}_X (\gamma)$  is a quasi--polynomial on certain regions
$\mathfrak c-B(X)$,  where $B(X):=\{\sum_{i=1}^mt_ia_i,\ 0\leq
t_i\leq 1\}$  is the {\it zonotope}  generated by $X$  while
$\mathfrak c$ denotes a {\it big cell}, that is a connected
component of the complement in $V$ of the {\it singular vectors}
which are formed by the union of all cones $C(Y)$  for all the
sublists $Y$ of $X$  which do not span $V$.  The  complement of  $C(X)$ is a big cell.  The other cells are inside $C(X)$ and are   convex.

The quasi--polynomials describing the partition function belong to
a remarkable finite dimensional space introduced and described by
Dahmen--Micchelli  \cite{DM} and which in this paper will be
denoted by $DM(X)$. This is the space of solutions of a system  of
difference equations. In order to describe it, let us   call a
subspace $\underline r$ of $V$ {\it rational} if $\underline r$ is
the span of a sublist of $X$. We need to recall that a {\it
cocircuit} $Y$ in $X$  is a sublist of $X$  such that $X\setminus
Y$ does not span $V$ and $Y$ is minimal with this property.  Thus
$Y$ is of the form $Y=X\setminus H$  where $H$ is a rational
hyperplane. Given $a\in\Gamma$, the {\it difference operator}
$\nabla_a$ is the operator on functions defined by
$\nabla_a(f)(b):=f(b)-f(b-a)$. For a list $Y$ of vectors, we set
$\nabla_Y:=\prod_{a\in
Y}\nabla_a$.
 We   then define:
$$DM(X):=\{f\,|\,  \nabla_Y
f=0,\ \\ \text{ for every  cocircuit $Y$ in }X\}.$$ It is easy to
see that $DM(X)$ is finite dimensional and consists
 of quasi--polynomial
functions  (cf. \cite{dp1}).

In this article,
we introduce
$$\mathcal F(X):=\{f\,|\,  \nabla_{X\setminus  \underline r }
f \text{ is supported on }  \underline r \text{ for every
 rational subspace    } \underline r  \}.
$$
Clearly ${\mathcal P }_X $ as well as $DM(X)$ are contained in
$\mathcal F(X)$. The space $\mathcal F(X)$ is of interest even when $X$ does not
span a pointed cone and occurs  in studying indices of
transversally elliptic operators on a vector space.

The main result of this article is  a ``localization formula'' for
an element $f$ in $\mathcal F(X)$. In particular, given a chamber
$\mathfrak c$, our localization formula allows us to write
explicitly the partition function ${\mathcal P }_X $ as a sum of a
quasi--polynomial function ${\mathcal P }_X ^{\mathfrak c}\in
DM(X)$ and of other functions $f_{\underline r} \in \mathcal F(X)$
supported outside $\mathfrak c-B(X)$.  This allows us to give a
short proof of the quasi--polynomiality of ${\mathcal P }_X $ on
the regions $\mathfrak c-B(X)$. Furthermore this
de\-com\-po\-si\-tion implies Paradan's wall crossing formulae
\cite{P1} for the quasi--polynomials ${\mathcal P }_X ^{\mathfrak
c}$. Our approach is strongly inspired by Paradan's localization
formula in Hamiltonian geometry,  but our methods here are
elementary.
 We wish to thank Michel Duflo  and Paul-Emile Paradan for several suggestions and corrections.

\section{Special functions}
\subsection{Basic notations}

Let $\Gamma$ be a lattice and $E=\mathbb Z,\mathbb Q, \mathbb R,\mathbb C$.
Consider the space ${\mathcal C}_E[\Gamma]$ of $E$ valued functions  on $\Gamma$.  When $E=\mathbb Z$,  we shall simply write ${\mathcal C}[\Gamma]$.
We display such a function $f(\gamma)$ also as a formal series
$$\Theta(f):=\sum_{\gamma\in\Gamma}f(\gamma) e^{ \gamma }.$$
Of course, under suitable convergence conditions, the  series
$\sum_{\gamma\in\Gamma}f(\gamma) e^{ \gamma } $ is a function on
the torus $G$  whose character group is $\Gamma$, and it is the
Laplace--Fourier transform of $f$.  In fact the functions that we
shall study   are Fourier coefficients of some important  generalized functions on $G$.
This fact and the several implications for the index theory of transversally elliptic operators  will be the subject of a subsequent paper \cite{dpv}.

The space ${\mathcal C}_E[\Gamma]$ is in an obvious way  a module over  the  group
algebra $E[\Gamma]$, multiplication by $e^\lambda$ on the
series $\Theta(f)$ corresponding to the translation operator
$\tau_\lambda$ defined by $$ (\tau_\lambda
f)(\gamma):=f(\gamma-\lambda)$$  on the function $f$. Thus
multiplication by $1-e^\lambda$ corresponds to the difference operator
$\nabla_\lambda $.

We denote by $\delta_0$ the function on $\Gamma$ identically equal
to $0$ on $\Gamma$,  except  for $\delta_0(0)=1$.  Remark that the
product of two formal series $\Theta(f_1)\Theta(f_2)$, whenever it
is defined, corresponds to {\it convolution} $f_1*f_2$ of the
functions $f_1$ and $f_2$. The function $\delta_0$ is the unit
element.\smallskip

\begin{remark}
Notice that, for a difference operator $\nabla_a$ acting on a
convolution, we have:
$$\nabla_a(f_1*f_2)=\nabla_a(f_1)*f_2=f_1*\nabla_a(f_2 ).$$
\end{remark}

Let now $X:=[a_1,\dots,a_m]$ be a list of non zero elements of
$\Gamma$ and let $V:=\Gamma\otimes_{\mathbb Z}\mathbb R$ be the
real vector space generated by $\Gamma$. We assume that $X$
generates the vector space $V$,  but we do not necessarily assume that $X$
generates a pointed cone in $V$.

If $X$ generates a pointed cone, then we can define
$$\Theta_X=\prod_{a\in X}\sum_{k=0}^\infty e^{ka}.$$
We write $$\Theta_X=\sum_{\gamma\in \Gamma}{\mathcal P }_X
(\gamma)e^{\gamma}$$ where  ${\mathcal P }_X \in {\mathcal C} [\Gamma]$ is the partition
function. `` Morally'' the series $\Theta_X$ is equal to
$\prod_{a\in X}\frac{1}{1-e^a}$, but $\frac{1}{1-e^a}$ has to be
understood  as the geometric series expansion $\sum_{k=0}
^{\infty}e^{ka}$.
\begin{remark}
We easily see that the partition
function satisfies the difference equation $\nabla_X{\mathcal P }_X =\delta_0$.
 Clearly this equation has infinitely many solutions.
 The fact that ${\mathcal P }_X $  is uniquely determined
 by the recursion expressed by this equation
 comes from the further property of this solution of
  {\it having support in the cone $C(X)$}.  We shall see other functions of the same type appearing in this paper.
\end{remark}

\begin{definition}\begin{enumerate}\item
A subspace of $V$ generated by a subset of the  elements of $X$
will be called {\bf rational  (relative to $X$)}. \item Given a
rational subspace $\underline r$,
 we denote  by ${\mathcal C}[\Gamma,
\underline r]$  the set of elements in ${\mathcal C}[\Gamma]$ which have
support  in the lattice $\Gamma\cap \underline r$.

\item
Given a rational subspace $\underline r$,
we set  $\nabla_{X\setminus\underline r}:=\prod_{a\in X\setminus \underline
r}\nabla_a.$
\end{enumerate}\end{definition}

With these notations, the space $DM_E(X)$  defined by
Dahmen--Micchelli   is formed  by the set of functions $f\in
{\mathcal C}_E[\Gamma]$  satisfying the system of difference equations
$\nabla_{X\setminus\underline r}f=0$ as ${\underline r}$ varies
among all proper  rational  subspaces relative to $X$. It is easy
to see that $ DM_E(X)$ consists of quasi--polynomials.

It  follows from their theory  (see also \cite{dp1})  that,  for each $E$, the space $DM_E(X)$ is a free $E-$module of dimension $\delta(X)$,  the volume of the zonotope $B(X)$.  In particular   $DM_E(X)=E\otimes _{\mathbb Z}DM_{\mathbb Z}(X)$ for all $E$.  Therefore from now on we shall work directly over $\mathbb Z$ and drop the subscript $E$.

The smallest  sub--lattice of $\Gamma$  for which each function of
$DM(X)$ is a polynomial  on its cosets  is the
 intersection  of all the sublattices of $\Gamma$ generated by  all
 the bases of $V$ that  one can extract from $X$  (the least common
 multiple).

Given a rational subspace $\underline r$, we  will identify  the
space ${\mathcal C}[\Gamma\cap \underline r]$  with the subspace ${\mathcal C}[\Gamma,
\underline r]$ of ${\mathcal C}[\Gamma]$  by extending the functions with 0
outside $\underline r$.

\subsection{The special functions  ${\mathcal P }_{X\setminus \underline r}^{F_{\underline r}}$.}
Given a rational subspace $\underline r$,    $X\setminus \underline r$ defines a hyperplane arrangement in the space   $\underline r^\perp\subset U$ orthogonal to $\underline r$. Take an open  face $F_{\underline r}$ in   $\underline r^\perp$ with respect to this hyperplane arrangement. We shall call such a face a regular face for $X\setminus \underline r$. If
$F_{\underline r}$ is a regular face, then $-F_{\underline r}$ is also a regular face.
  By definition a vector  $u\in \underline r^\perp$ and such that $\langle u, a\rangle \neq 0$ for all $a \in X\setminus \underline r$
lies in a unique such regular face $F_{\underline r}$ and $u$  will be called a regular vector for $X\setminus\underline r$.

Given a regular face $F_{\underline r}$ for
$X\setminus \underline r$,  we divide $X\setminus \underline r$ into two parts $A,B$,   according to whether they take positive or negative values on our face. We denote the cone  $C(A,-B)$,   generated
by the list  $[A,-B]$, by $C(F_{\underline r},X\setminus\underline r)$.

 If we take $u\in F_{\underline r}$,  $C(F_{\underline r},X\setminus\underline r)$ is contained in the
closed half space of vectors $v$ where $u$ is non negative.

We are going to consider the series $\Theta_{X\setminus\underline r}^{F_{\underline r}}$ which is
characterized by the following two properties:
\begin{lemma}\label{lafunteta}
There exists a unique element  $$\Theta_{X\setminus\underline r}^{F_{\underline r}}
=\sum_{\gamma}{\mathcal P }_{X\setminus\underline r}^{F_{\underline r}}(\gamma)e^{\gamma}$$ such that
\begin{enumerate}\item
$\prod_{a\in X\setminus \underline r}(1-e^a) \Theta_{X\setminus\underline r}^{F_{\underline r}}=1,$  equivalently
$\nabla_{X\setminus \underline r}{\mathcal P
}_{X\setminus\underline r}^{F_{\underline r}}=\delta_0$.\item $\mathcal P_{X\setminus\underline r}^{F_{\underline r}}$ is supported in
$-\sum_{b\in B}b+ C(F_{\underline r},{X\setminus\underline r}).$
\end{enumerate}
\end{lemma}
\begin{proof}  Set:
\begin{equation}\label{esplicit}
\Theta_{X\setminus\underline r}^{F_{\underline r}}=(-1)^{|B|}e^{-\sum_{b\in B}b}\prod_{a\in
A}(\sum_{k=0}^\infty e^{ka}) \prod_{b\in B}(\sum_{k=0}^\infty
e^{-kb}) .\end{equation} It  is easily seen that this element
satisfies the two properties and is unique.
\end{proof}

In particular, if $\underline r=V$,  we have $F_V=\{0\}$ and $\mathcal P^{\{0\}}_X=\delta_0$.

Morally, $\Theta_{X\setminus\underline r}^{F_{\underline r}}=\prod_{a\in X\setminus \underline
r}\frac{1}{1-e^a}=\prod_{a\in A}\frac{1}{1-e^a}\prod_{b\in
B}\frac{-e^{-b}}{1-e^{-b}}.$ We indeed need to reverse the sign of
some of the vectors in $X\setminus \underline r$ in order that the
convolution product of the corresponding geometric series makes
sense.

\bigskip

Although a function $f\in {\mathcal C}[\Gamma, \underline r]$  may have
infinite support,
 we easily see that the convolution  ${\mathcal P }_{X\setminus\underline r}^{F_{\underline r}}*f$
 is well defined. In fact we claim that, given any $\gamma\in \Gamma$, we can
write $\gamma=\lambda+\mu$ with $\mu\in \underline r\cap \Gamma, $
and $\lambda\in (-\sum_{b\in B}b+C(A,-B))\cap \Gamma$ only in
finitely many ways.  To see this, take $u\in F_{\underline r}$. Then  $\langle u\,|\,
\gamma\rangle=\langle u\,|\, \lambda\rangle$ and $\lambda
=\sum_{a\in A} k_a a+\sum_{b\in B} h_b(-b)$ with $k_a\geq 0,\
h_b\geq 1$. Thus the equality  $\langle u\,|\,
\gamma\rangle=\sum_{a\in A} k_a \langle u\,|\, a\rangle+\sum_{b\in
B} h_b\langle u\,|\, -b\rangle$  yields that the vector $\lambda$
is in a bounded set,  intersecting the lattice $\Gamma$ in a
finite set.

\bigskip

Choose two rational spaces $\underline r,  \underline t$ and     an regular  face $F_{\underline r}$ for $X\setminus \underline r$ in   $\underline r^\perp$.

The image  of  $F_{\underline r}$ modulo  $\underline t^\perp$ is a regular  face  for $(X\cap \underline t)
\setminus \underline r$. To simplify notations, we still denote this face by
$F_{\underline r}$. We have:

\begin{proposition}\label{promother}
\begin{enumerate}
\item $\nabla_{(X\setminus \underline t)\setminus \underline r }
{\mathcal P }_{X\setminus\underline r}^{F_{\underline r}}={\mathcal P }_{(X\setminus\underline r)\cap \underline t}^{F_{\underline r}}$.

\item For $g\in {\mathcal C}[\Gamma\cap \underline r]$:
\begin{equation}\label{eqmother}
\nabla_{X\setminus \underline t} ({\mathcal P }_{X\setminus\underline r}^{F_{\underline r}}*g) ={\mathcal
P }_{(X\setminus\underline r)\cap \underline t}^{F_{\underline r}}* (\nabla_{(X\cap \underline
r)\setminus (\underline t \cap \underline r)}g).
\end{equation}
\end{enumerate}
\end{proposition}

\begin{proof}
{\it i)} From Equation (\ref{esplicit}),
  we see that  the series associated to the function
  $\nabla_{(X\setminus \underline t)\setminus \underline r }
  {\mathcal P }_{X\setminus\underline r}^{F_{\underline r}}$
  equals
$$
\Theta_{(X\setminus\underline r)  \cap \underline t}^{F_{\underline r}}=(-1)^{|B\cap \underline
t|}e^{-\sum_{b\in B\cap \underline t}b}\prod_{a\in A\cap
\underline t}(\sum_{k=0}^\infty e^{ka}) \prod_{b\in B\cap
\underline t}(\sum_{k=0}^\infty e^{-kb}) .$$

{\it ii)} Let $g\in {\mathcal C}[\Gamma\cap \underline r]$.  Take any
rational subspace $\underline t$, we have that
 $\nabla_{X\setminus \underline t} =
 \nabla_{(X\cap \underline r)\setminus (\underline t \cap \underline r)}
 \nabla_{(X\setminus \underline t)\setminus \underline r }$, thus
$$\nabla_{X\setminus \underline t}( {\mathcal P }_{X\setminus\underline r}^{F_{\underline r}}*g) =
( \nabla_{(X\setminus \underline t)\setminus \underline r
}{\mathcal P }_{X\setminus\underline r}^{F_{\underline r}})* (\nabla_{(X\cap \underline r)\setminus
(\underline t \cap \underline r)}g).$$

 As
$\nabla_{(X\setminus \underline t)\setminus \underline r }
{\mathcal P }_{X\setminus\underline r}^{F_{\underline r}}={\mathcal P }_{(X\setminus\underline r)\cap  \underline t}^{F_{\underline r}}$ from part
{\it i)}, we obtain Formula (\ref{eqmother}), which is the mother
of all other formulae of this article.

\end{proof}

In particular, for $\underline r=\underline t$,   Formula (\ref{eqmother}) implies
the following.
\begin{proposition}\label{proar0}
 If $f\in {\mathcal C}[\Gamma\cap  \underline r]$,
  we have $f=\nabla_{X\setminus \underline r}({\mathcal P }_{X\setminus\underline r}^{F_{\underline r}}*f)$.
  \end{proposition}

 \section{A remarkable space}

 \subsection{The space $\mathcal F(X)$.}
We let $S_X$ denote the set of all rational subspaces
relative to $X$.
\begin{definition}\label{spin}
We   define the space of interest for this article  by:
\begin{equation}
\mathcal F(X):=\{f\in {\mathcal C}[\Gamma]\,|\, \nabla_{X\setminus\underline
r}f\in {\mathcal C}[\Gamma, \underline r], \text{ for all } \underline r\in
S_X\}.
\end{equation}
\end{definition}
One of the equations (corresponding to   $\underline r=\{0\}$) that must satisfy $\Theta(f)$ when $f\in \mathcal F(X)$ is the relation $\prod_{a\in X}(1-e^a)\Theta(f)=c$, where $c$ is a constant. This
equation was the motivation for introducing the space  $\mathcal F(X)$.
Indeed the first  important fact on  this space is the following:
\begin{lemma}\label{lacomin}\begin{enumerate}
\item If $F$ is a  regular face for   $X$,  then
$\mathcal P_X^F $ lies in $\mathcal F(X)$. \item The space $DM(X)$
is contained in   $\mathcal F(X)$.
\end{enumerate}\end{lemma}
\begin{proof}
{\it i)} Indeed,   $\nabla_{X\setminus\underline r} {\mathcal P
}_X^F =\mathcal P_{X\cap\underline r}^F\in {\mathcal C}[\Gamma,\underline r]$.

{\it ii)}  Is clear from the definitions.
\end{proof}

In particular,  if $X$ generates a pointed cone, then the partition function
$\mathcal P _X $ lies in $\mathcal F(X)$.

\begin{example} Let us give a  simple example. Let $\Gamma=\mathbb
Z\omega$ and $X=[2\omega,-\omega]$. Then it is easy to see that
$\mathcal F(X)$ is a free $\mathbb Z$ module of dimension $4$, with corresponding basis
$$\theta_1=\sum_{n\in \mathbb Z}e^{n\omega},\hspace{0.5cm}
\theta_2=\sum_{n\in \mathbb Z}ne^{n\omega},$$ $$
\theta_3=\sum_{n\in \mathbb Z} (\frac{n}{2}+\frac{1-(-1)^n}{4})e^{n\omega},\hspace{0.5cm}
\theta_4=\sum_{n\geq 0}
(\frac{n}{2}+\frac{1-(-1)^n}{4})e^{n\omega}.$$

Here $\theta_1,\theta_2,\theta_3$ are a $\mathbb Z$ basis of $DM(X)$.

\end{example}

In fact, there is a much more precise statement  of which Lemma
\ref{lacomin} is  a very special case and which will be the object
of Theorem \ref{gesta}.

\subsection{Some properties of $\mathcal F(X)$.}

Let $\underline r$ be a rational subspace and $F_{ \underline
r }$ be a regular face for $X\setminus \underline r$.

\begin{proposition}\label{proar}

\begin{enumerate}
 \item  The map $g\mapsto
{\mathcal P }_{X\setminus\underline r}^{F_{\underline r}}*g$  gives an injection from $\mathcal F(X\cap
\underline r)$ to $\mathcal F(X)$. Moreover
\begin{equation}\label{rinve}
\nabla_{X\setminus \underline r} ({\mathcal P }_{X\setminus\underline r}^{F_{\underline r}}*g)=g,\ \forall g\in  \mathcal F(X\cap
\underline r).
\end{equation}  \item $\nabla_{X\setminus \underline r}$ maps $\mathcal F(X)$ surjectively to
$\mathcal F(X\cap \underline r)$.\item If $g\in DM(X\cap
\underline r)$, then $\nabla_{X\setminus \underline t} ({\mathcal
P }_{X\setminus\underline r}^{F_{\underline r}}*g)=0$ for any rational subspace $\underline t$ such that
$\underline t\cap \underline r\neq \underline r$.

\end{enumerate}
\end{proposition}

\begin{proof}
{\it i)} If $g\in\mathcal F(X\cap \underline r)$, then
$\nabla_{(X\cap \underline r)\setminus (\underline t \cap
\underline r)}g\in {\mathcal C}[\Gamma\cap \underline t \cap \underline r]$,
hence Formula (\ref{eqmother}) in Proposition \ref{promother}
shows that $\nabla_{X\setminus \underline t} ({\mathcal P }_{X\setminus\underline r}^{F_{\underline r}}*g)
\in {\mathcal C}[\Gamma,\underline t]$, so that ${\mathcal P }_{X\setminus\underline r}^{F_{\underline r}}*g\in
\mathcal F(X)$
  as desired.  Formula (\ref{rinve})  follows from the fact that $\nabla_{X\setminus \underline r}  {\mathcal P }_{X\setminus\underline r}^{F_{\underline r}}=\delta_0$.
\smallskip

{\it ii)}  If $f\in \mathcal F(X)$, we have $\nabla_{X\setminus
\underline r}f\in \mathcal F(X\cap \underline r)$. In fact  take a
rational subspace $\underline t$ of $\underline r$, we have that
$\nabla_{(X\cap \underline r)\setminus \underline
t}\nabla_{X\setminus \underline r}f=\nabla_{X\setminus \underline
t}f\in {\mathcal C}[\Gamma\cap \underline t]$.  The fact that $\nabla_{X\setminus
\underline r}$ is surjective is a consequence of  Formula (\ref{rinve}).
\smallskip

{\it iii)} Similarly, if $g\in  DM(X\cap \underline r)$,
Formula (\ref{eqmother}) in Proposition \ref{promother} implies
the third assertion of our proposition.

\end{proof}

Proposition \ref{proar} allows us to  associate,  to a rational
space ${\underline r}$ and a regular face $F_{\underline r}$ for $X\setminus \underline r$, the operator
$${\Pi}_{X\setminus\underline r}^{\underline r, F_{\underline r}}:f\mapsto {\mathcal P}_{X\setminus\underline r}^{F_{\underline r}}*
(\nabla_{X\setminus {\underline r} }f)$$ on  $\mathcal F(X)$. From  Formula (\ref{rinve}),   it follows that the
operator ${\Pi}_X^{\underline r, F_{\underline r}}$, on  $\mathcal F(X)$, is a projector with image ${\mathcal P}_{X\setminus\underline r}^{F_{\underline r}}*\mathcal F(X\cap
\underline r).$

\subsection{The main theorem}
Choose,  for every rational space $\underline r$, a   regular face
$F_{\underline r}$ for $X\setminus    {\underline r}$. The following theorem is the main
theorem of this section.
\begin{theorem}\label{gesta}
With the previous choices,  we have:
\begin{equation}\label{decomp}
\mathcal F(X)=\oplus_{\underline r\in S_X} {\mathcal
P}_{X\setminus\underline r}^{F_{\underline r}}*DM(X\cap \underline r ).
\end{equation}
\end{theorem}

\begin{proof}
Denote by  $
S^{(i)}_X$ the subset  of  subspaces  $\underline r\in S_X $ of  dimension $i$.
Consider   $\nabla_{X\setminus\underline r}$ as an operator on
$\mathcal F(X)$ with values in ${\mathcal C}[\Gamma]$.
 Define the spaces
$$\mathcal F(X)_{i}:=\cap_{\underline t\in S^{(i-1)
}_X}\ker(
 \nabla_{X\setminus \underline t}).$$

Notice that   by definition $\mathcal F(X)_{\{0\}}=\mathcal F(X)$, that
$\mathcal F(X)_{\dim V}$ is the space  $DM(X)$ and that $\mathcal F(X)_{i+1}\subseteq
\mathcal F(X)_{i}$.

\begin{lemma}\label{lalari}
Let $\underline r \in {S}_X^{(i)}$.

\quad i)
 The image of $\nabla_{X\setminus\underline r}$ restricted
to $\mathcal F(X)_{i}$ is contained in the  space   $DM(X\cap \underline r)$.

\quad ii) If $f$ is in  $DM(X\cap \underline r)$, then
${\mathcal P }_{X\setminus\underline r}^{F_{\underline r}}*f \in \mathcal F(X)_{i}$.
\end{lemma}
\begin{proof}
{\it i)} First we know, by the definition of $\mathcal F(X)$, that
$\nabla_{X\setminus\underline r} \mathcal F(X)_{i}$ is contained in the space
${\mathcal C}[\Gamma, \underline r]$. Let $\underline t$  be a rational
hyperplane of $\underline r$, so that $\underline t$ is of
dimension $i-1$. By construction, we have that for every $f\in
\mathcal F(X)_{i}$
$$0=\prod_{a\in X\setminus \underline t}\nabla_a f=
\prod_{a\in (X\cap \underline r)\setminus \underline t} \nabla_a
\nabla_{X\setminus\underline r}f.$$ This means that
$\nabla_{X\setminus\underline r} f$ satisfies the difference
equations given by the cocircuits of   $X\cap \underline r$, that is,  it lies in $DM(X\cap \underline r)$.

\smallskip

{\it ii)} Follows from the third item of Proposition \ref{proar}.

\end{proof}

Consider the map $\mu_i: \mathcal F(X)_i \to \oplus_{\underline r\in
S_X^{(i)}} DM(X\cap \underline r)$ given by $$\mu_i
f:=\oplus_{\underline r\in S_X^{(i)}} \nabla_{X\setminus
\underline r}f$$ and the map ${\bf P}_i:\oplus_{\underline r\in
S_X^{(i)}} DM(X\cap \underline r)\to \mathcal F(X)_i$ given by
$${\bf P}_i(\oplus g_{\underline
r}):= \sum \mathcal P_{X\setminus\underline r}^{F_{\underline r}}*g_{\underline r }.$$

\begin{theorem}\label{lemexact}
The sequence $$0\longrightarrow \mathcal F(X)_{i+1} \longrightarrow \mathcal F(X)_i
\stackrel{\mu_i}\longrightarrow\oplus_{\underline r\in
S_X^{(i)}} DM(X\cap \underline r)\longrightarrow 0$$ is
exact. Furthermore, the map ${\bf P}_i$ provides a splitting  of
this exact sequence: $\mu_i {{\bf P}_i}={\rm Id}$.
\end{theorem}
\begin{proof}
By definition, $\mathcal F(X)_{i+1}$ is the kernel of $\mu_i$,  thus we only need to show that
$\mu_i {{\bf P}_i}={\rm Id}$.
Given   $\underline r\in S_X^{(i)}$ and $g\in DM(X\cap \underline r)$,  by Formula (\ref{rinve}) we have  $\nabla_{X\setminus \underline r}({\mathcal P }_{X\setminus\underline r}^{F_{\underline r}}*g)=g.$  If instead we take $
\underline t\neq \underline r$ another subspace  of $S_X^{(i)}$,
$\underline r\cap \underline t$ is a proper subspace of
$\underline t$. Item {\it iii)} of Proposition \ref{proar} says
that for $g\in DM(X\cap \underline r)$, $\nabla_{X\setminus
\underline t} ({\mathcal P }_{X\setminus\underline r}^{F_{\underline r}}*g) =0.$ Thus,  given a family $g_{\underline r}\in DM(X\cap \underline r)$, the
function $f= \sum_{\underline t\in S_X^{(i)}} {\mathcal P
}_{X\setminus\underline t}^{F_{\underline t}}*g_{\underline t}$ is such that
$\nabla_{X\setminus \underline r}f=g_ {\underline r}$ for all $\underline r\in  S_X^{(i)}$. This proves our claim
that  $\mu_i {{\bf P}_i}={\rm Id}$.\end{proof}

Putting together these   facts, Theorem \ref{gesta} follows.

\end{proof}

\bigskip

A collection ${\bf F}=(F_{\underline r})$ of   faces
$F_{\underline r}\subset \underline r^\perp$ regular for $X\setminus \underline r$, indexed by the rational subspaces $\underline r\in S_X$
 will be called a {\it  $X$--regular collection}.

Given a $X$-regular collection ${\bf F}$, we can write an element
$f\in \mathcal F(X)$ as
$$f=\sum_{\underline r\in S_X}f_{\underline
r},\qquad \text{(Theorem \ref{gesta})}$$ with $f_{\underline r}\in {\mathcal P}_{X\setminus\underline r}^{F_{\underline
r}}*DM(X\cap \underline r ).$  This expression for $f$ will be
called  the ${\bf F}$ de\-com\-po\-si\-tion of $f$.
 In this
de\-com\-po\-si\-tion, we always have $F_V=\{0\},\, {\mathcal P}_X^{F_{V}}=\delta_0$  and the component $f_V$ is in $DM(X).$

The space ${\mathcal P}_{X\setminus\underline r}^{F_{\underline r}}*DM(X\cap \underline r
)$
 will be referred to as {\it the $F_{\underline r}$-component} of $\mathcal F(X)$.

\medskip

From Lemma \ref{lemexact},  it follows that  the operator  ${\rm Id}-{{\bf P}_i}\mu_i $  projects $\mathcal F(X)_i$  to $\mathcal F(X)_{i+1}$  with kernel $\oplus_{\underline r\in
S_X^{(i)}} \mathcal P_{X\setminus\underline r}^{F_{\underline r}}*DM(X\cap \underline r)$ (this operator depends of ${\bf F}$).
 Thus  the ordered product
$${\bf \Pi}_i^{\bf F}:= ({\rm Id}-{{\bf P}_{i-1}}\mu_{i-1} )({\rm Id}-{{\bf P}_{i-2}}\mu_{i-2} )\cdots ({\rm Id}-{{\bf P}_0}\mu_0 )$$  projects
$\mathcal F(X)$  to $\mathcal F(X)_{i}$; therefore,  we have

\begin{proposition}\label{gliop}
Let ${\bf F}$  be a $X$-regular collection and $ \underline r $  a rational subspace of dimension $i$.
 The operator
$$P_{\underline r}^{\bf F}= \Pi_X^{\underline r,F_{\underline r}}
({\rm Id}-{{\bf P}_{i-1}}\mu_{i-1} )({\rm Id}-{{\bf P}_{i-2}}\mu_{i-2} )\cdots ({\rm Id}-{{\bf P}_0}\mu_0 )=\Pi_X^{\underline r,u_{\underline r}}{\bf \Pi}_i^{\bf u}$$ is the projector of
$\mathcal F(X)$ to the   $F_{\underline r}$-component  $ {\mathcal
P}_{X\setminus\underline r}^{F_{\underline r}}*DM(X\cap \underline r )$ of $\mathcal
F(X)$.

 In particular, if $\dim(V)=s$,  the operator
$$P_V:=({\rm Id}-{{\bf P}_{s-1}}\mu_{s-1} )({\rm Id}-{{\bf P}_{s-2}}\mu_{s-2} )\cdots ({\rm Id}-{{\bf P}_0}\mu_0 )$$ is the projector
$\mathcal F(X)\to DM(X)$ associated to the direct sum
de\-com\-po\-si\-tion:
$$\mathcal F(X)=DM(X)\oplus \Big(\oplus_{\underline r\in S_X| {\underline r}
\neq V} {\mathcal P}_{X\setminus\underline r}^{F_{\underline r}}*DM(X\cap \underline r
)\Big).$$
\end{proposition}

Let ${\bf F}=(F_{\underline r})$  be a $X$--regular collection. If
$\underline t$ is a rational subspace  and, for each $\underline r\in
S_{X\cap \underline t}$, we take the image of $F_{\underline r}$ modulo $\underline t^{\perp}$ we get a    $X\cap \underline t$--regular
collection. We still denote it by $\bf F$ in the next proposition.
The proof of this proposition is skipped, as it is very similar to
preceding   proofs.

\begin{proposition}\label{recurrence}
Let $\underline t$ be a rational subspace.
 Let $f\in \mathcal F(X)$ and $f=\sum_{\underline r\in
S(X)}f_{\underline r}$ be the ${\bf F}$ de\-com\-po\-si\-tion of
$f$ and $\nabla_{X\setminus \underline t}f=\sum_{{\underline t}\in
S_{X\cap \underline t} } g_{\underline r}$ be the ${\bf
F}$ de\-com\-po\-si\-tion of $\nabla_{X\setminus \underline t}f$,
then

\begin{itemize}
\item $\nabla_{X\setminus s}f_{\underline r}=0$ if $\,{\underline
r}\notin S_{X\cap \underline t},$ \item
$\nabla_{X\setminus\underline  s}f_{\underline r}= g_{\underline
r}$ if $\,\underline r\in S_{X\cap \underline t} $.
\end{itemize}
\end{proposition}

\begin{remark}
It follows from the previous theorems and  the properties of $DM(X)$  that, for every  $E$, we could define a space $\mathcal F_E(X)$ of $E$ valued functions  as in Definition \ref{spin}  and we have  $\mathcal F_E(X)=E\otimes_{\mathbb  Z}\mathcal F (X).$
\end{remark}

\subsection{Localization theorem\label{211}} \begin{definition}
By the word {\it
tope},  we mean   a connected component of the complement in $V$
 of the union of
the hyperplanes generated by subsets of $X$.
\end{definition}
We introduce also  $B(X):=\{\sum_{i=1}^mt_ia_i,\ 0\leq t_i\leq 1\}$   the {\it
zonotope}  generated by $X$. $B(X)$ is a compact convex polytope which appears  in several ways in the theory and plays a fundamental role.

  We will show that,  for
every element $f\in \mathcal F(X)$, the function  $f(\gamma)$
coincides with  a quasi--polynomial on the sets $(\tau -B(X))\cap
\Gamma$ as $\tau$ varies over all topes (we simply say $f$ is a quasi--polynomial on $\tau-B(X)$).

  \begin{definition}
Let $\tau$ be a tope and
  $\underline r$ be a proper rational subspace. We say that
  a regular face $F_{\underline r}$ for $X\setminus \underline r$
  is non--positive on $\tau $ if there exists  $u_{\underline r}\in F_{\underline r}$ and $x_0\in \tau$ such that $\langle u, x_0\rangle <0$.
  \end{definition}

Given   $x_0\in \tau$,
 it
is always  possible to choose a regular face $F_{\underline r}\subset \underline r^{\perp}$ for $X\setminus \underline r$ such
that  $x_0$ is negative on some vector $u_{\underline r}\in F_{\underline r}$, since  the projection of $x_0$ on
$V/{\underline r}$ is not zero.

 Let  ${\bf F}=\{ F_{\underline r}\}$ be     a $X$--regular  collection.  We shall say that $\{\bf F\}$ is non--positive on $\tau$ if each $F_{\underline r}$ is  non--positive on $\tau$.

Let $f\in \mathcal F(X)$ and let
 $f=\sum f_{\underline
r}$ be the ${\bf F}$ de\-com\-po\-si\-tion of $f$.

\begin{remark}\label{abc}
This choice of ${\bf F}$  has
the effect of pushing the supports of the elements $f_{ \underline
r}$ ($\underline r\neq V$) away from  $\tau$.  See Figures \ref{uu}, \ref{uuu} which describe the ${\bf
F}$ decomposition of the partition function $\mathcal P_X$ for
$X:=[a,b,c]$  with $a:=\omega_1,b:=\omega_2,c:=\omega_1+\omega_2$
in the lattice $\Gamma:=\mathbb Z \omega_1\oplus \mathbb Z
\omega_2$. Thus the content of Theorem \ref{larcarsexplicit}
is very similar to Paradan's localization theorem \cite{par2}.
\end{remark}

Our previous claim then
follows from the explicit construction below.

\begin{theorem} [Localization theorem]\label{larcarsexplicit}
Let $\tau$ be a tope. Let ${\bf F}=\{ F_{\underline r}\}$ be  a $X$--regular  collection  non--positive  on $\tau$.

 The
component $f_V$ of the ${\bf F}$ de\-com\-po\-si\-tion $f=\sum_{\underline r\in S_X} f_{\underline
r}$ is a
quasi--polynomial function in $DM(X)$ such that $f=f_V$ on
$\tau-B(X)$.
\end{theorem}
 \begin{figure}[!h]
\begin{center}\includegraphics{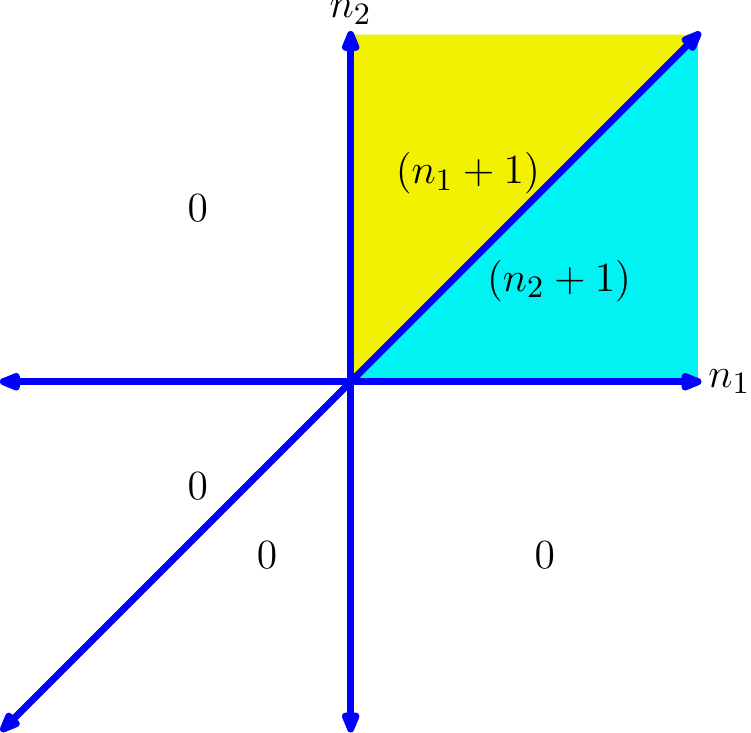}
\end{center}\caption{The partition function of $X:=[a,b,c]$}\label{uu}
\end{figure}

\begin{figure}[!h]
\begin{center}\includegraphics{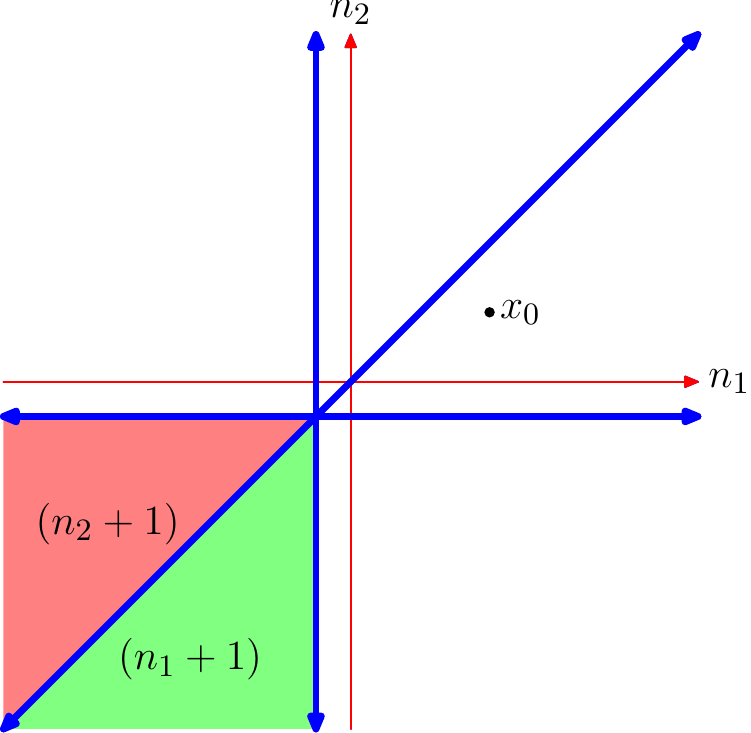}
\end{center}
$$+$$

\hskip-0.9cm \includegraphics{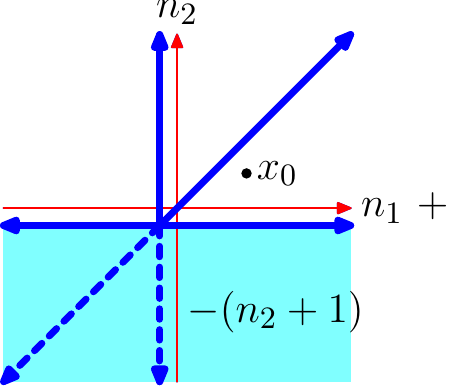} \includegraphics{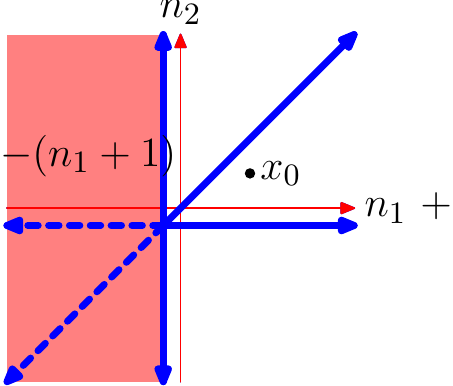} \includegraphics{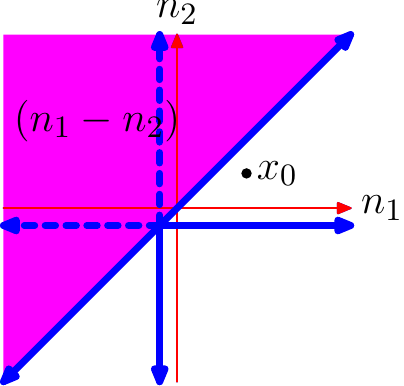}

$$+$$

\begin{center}\includegraphics{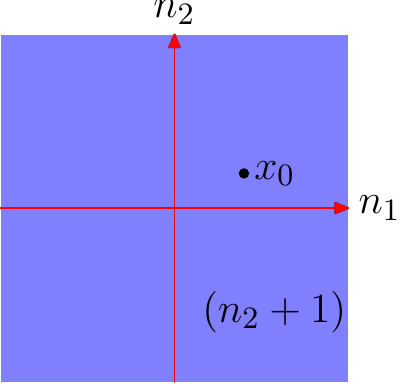}
\end{center}

\caption{ $\bf{F}$ decomposition of the partition function of
$X:=[a,b,c]$ for $\bf{F}$  non--positive  on $\tau$}\label{uuu}
\end{figure}
\begin{proof}

 Let
$\underline r$ be a proper rational space and $f_{\underline r}={\mathcal P}_{X\setminus\underline r}^{F_{\underline r}}* k_{\underline
r}$ where $k_{\underline r}\in DM(X\cap {\underline r})$.   In the notation of Lemma
\ref{lafunteta}, the support of $f_{\underline r}$ is contained in
the polyhedron ${\underline r}-\sum_{b\in B}b+C(F_{\underline
r},{X\setminus\underline r})\subset {\underline r}+C(F_{\underline r},{X\setminus\underline r})$.  This last
polyhedron is convex and, by construction, it has a boundary
limited by hyperplanes which are rational with respect to $X$.
Thus,  either $\tau\subset {\underline r}+C(F_{\underline r},{X\setminus\underline r})$ or
$\tau\cap ({\underline r}+C(F_{\underline r},{X\setminus\underline r}))=\emptyset$. Take $u_{\underline r}\in F_{\underline r}$  and $x_0\in \tau$ so that   $u_{\underline r}(x_0)<0$. As
$u_{\underline r}\geq 0$ on ${\underline r} +C(F_{\underline r}
,{X\setminus\underline r})$   it follows that $\tau$ is not
a subset of ${\underline r} +C(F_{\underline r} ,{X\setminus\underline r})$, so that
$\tau\cap ({\underline r} +C(F_{\underline r} ,{X\setminus\underline r}))=\emptyset$.

In fact, we claim that $\tau-B(X)$ does not intersect the support
${\underline r}-\sum_{b\in B}b+C(F_{\underline r},{X\setminus\underline r}) $ of
$f_{\underline r}$. Indeed, otherwise,  we would have an equation
$v-\sum_{x\in X}t_x x=s+\sum_{a\in A} k_a a+\sum_{b\in B} h_b(-b)$
with $v\in \tau$, $0\leq t_x\leq 1$, $k_a\geq 0,\ h_b\geq 1,\, s\in \underline r$. This
would imply that $v\in \underline r+C(F_{\underline r} ,{X\setminus\underline r}))$,  a contradiction. Thus
$f$ coincides with the quasi--polynomial $f_V$ on $\tau-B(X)$.
\end{proof}
One should remark that a quasi--polynomial is completely determined by the values  that it takes on $\tau-B(X)$,   thus $f_V$ is independent on the construction.
\begin{definition}
We shall denote by $f^\tau$ the quasi--polynomial coinciding with
$f$ on $\tau-B(X)$.
\end{definition}

Let us remark that the open subsets $\tau-B(X)$ cover $V$, when
$\tau$ runs over the topes of $V$ (with possible overlapping).
Thus the element $f\in \mathcal F(X)$ is entirely determined by
the quasi--polynomials $f^{\tau}$.

\begin{example}
In Figure \ref{a2},  for each tope $\tau$,  the set of integral points
in $\tau-B(X)$  is contained in  one of the  affine {\bf closed} cones limited by thick
lines.  We are   showing in color the convex envelop  of the integral points in    $\tau-B(X)$ and {\it not} the  larger  open set $\tau-B(X)$.

\begin{figure}[!h]
\begin{center}\includegraphics{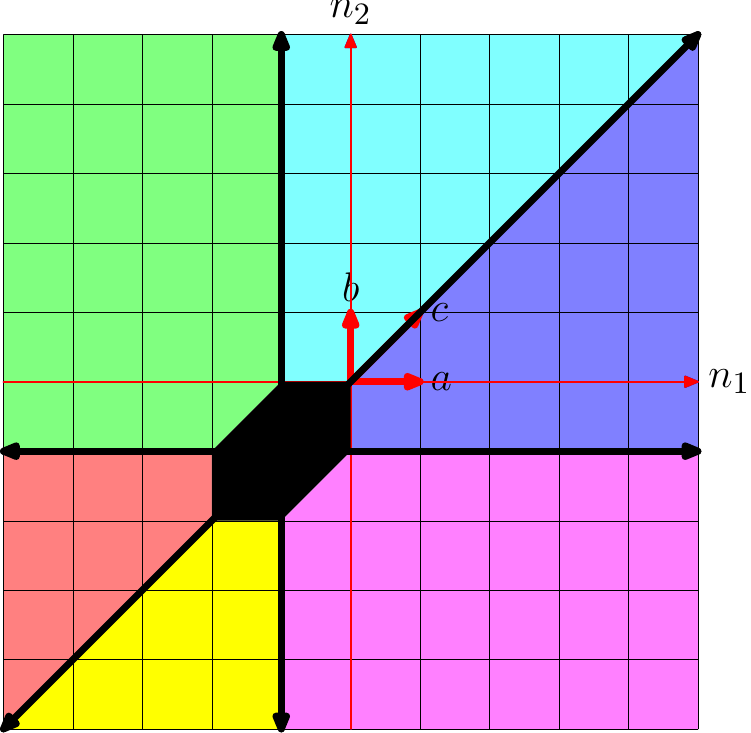}
\end{center}\caption{Translated topes of $X$. The zonotope $-B(X)$ is in black.}\label{a2}
\end{figure}

\end{example}

For  collections arising from scalar products, one can give an explicit formula for the   decomposition of an element $f\in \mathcal F(X)$.

Let us choose a scalar product on $V$, and identify $V$ and $U$ with respect to this scalar product.
Given a point $\beta$ in $V$  and a rational subspace $\underline r$ in $S_X$, we
write  $$\beta=p_{\underline r}\beta+p_{\underline r^{\perp}}\beta$$
with
$p_{\underline r}\beta\in \underline r$
and $p_{\underline r^{\perp}}\beta$ in $\underline r^{\perp}.$

\begin{definition}
We say that $\beta\in V$ is   generic with respect to $\underline r$ if
$p_{\underline r}\beta$  is in a tope  $\tau(p_{\underline r}\beta)$
for the sequence   $X\cap \underline r$,
and  $p_{\underline r^{\perp}}\beta \in
\underline r^{\perp}$  is regular for $X\setminus \underline r$.

\end{definition}
We clearly have
\begin{proposition}
The set of $\beta$ which are not generic with respect to $\underline r$ is a union of finitely many hyperplanes.
\end{proposition}

By Theorem  \ref{larcarsexplicit},   if $f\in \mathcal F(X)$, the element $\nabla_{X\setminus \underline r}f$ is in
$\mathcal F(X\cap \underline r)$ and coincides with a quasi--polynomial
$(\nabla_{X\setminus \underline r}f)^{\tau}\in DM(X\cap \underline r)$ on  each tope $\tau$
for the system $X\cap \underline r$.

\begin{theorem}\label{before}
 Let  $\beta\in V$ be
 generic   with respect to all the rational subspaces $\underline r$. Let $F_{\underline r}^\beta$ be the unique regular face for $X\setminus \underline r$ containing $p_{\underline r^{\perp}}\beta$.

 Then
$$f=\sum_{\underline r\in S_X} {\mathcal
P}_{X\setminus\underline r}^{-F_{\underline r}^\beta}*(\nabla_{X\setminus \underline r}f)^{\tau(p_{\underline r}\beta)}.$$
\end{theorem}

\begin{proof}
By the hypotheses made on $\beta$ the collection ${\bf F}=\{F_{\underline r}\}$ is  a $X-$regular collection.   Set
 $f= \sum_{\underline r\in S_X} {\mathcal
P}_{X\setminus\underline r}^{F_{\underline r}}*q_{\underline r}$
with $q_{\underline r}\in DM(X\cap \underline r)$.
We apply Proposition \ref{recurrence}. It follows that
$q_{\underline r}$ is the component in  $DM(X\cap \underline r)$ in the decomposition of
$\nabla_{X\setminus \underline r}f\in \mathcal F(X\cap \underline r)$, with
respect of the $X\cap \underline r$-regular collection induced by ${\bf F}$.
  Set $u_{\underline t}=-p_{\underline t^{\perp}}\beta$.
Remark that,  for $\underline t\subset \underline r$, we have
 $\langle u_{\underline t},p_{\underline r}\beta\rangle=-\|u_{\underline t }\|^2$, so that  each $u_{\underline t}$ is negative on $p_{\underline r}\beta$. Thus the formula follows from
Theorem \ref{larcarsexplicit}.
\end{proof}

\subsection{Wall crossing formula}

We first develop a general formula describing how the functions
$f^{\tau}$  change when crossing a wall. Then we apply this to the
partition function ${\mathcal P }_X $ and deduce that it is a
quasi--polynomial on $\mathfrak c-B(X)$, where $\mathfrak c$ is a
big cell.

Let $H$ be a rational hyperplane, and let $u\in
U$ be an equation of the hyperplane. Then the two open faces in $H^{\perp}$ are the half-lines
$F_H=\mathbf R_{>0}u$ and $-F_H$.

\begin{lemma}
If $q\in DM(X\cap H)$, then  $w:=(\mathcal P_{X\setminus H}^{F_{H}}- \mathcal
P_{X\setminus H}^{-F_{H}})*q$ is an element of $DM(X)$.
\end{lemma}

\begin{remark}
In \cite{BV}, a one-dimensional residue formula is given for $w$
allowing us to compute it.
\end{remark}

\begin{proof}
If ${\underline t}\in S_X$ is different from $H$,
$\nabla_{X\setminus {\underline t} }(\mathcal
P_{X\setminus H}^{F_{H}}*q)=\nabla_{X\setminus {\underline t} }(\mathcal
P_{X\setminus H}^{-F_{H}}*q)=0$, as  follows from Proposition \ref{proar} {\it
iii)}. If ${\underline r} =H$, then $\nabla_{X\setminus H}
(\mathcal P_{X\setminus H}^{F_{H}}*q-\mathcal P_{X\setminus H}^{-F_{H}}*q)=q-q=0$.
\end{proof}

Assume that $\tau_1,\ \tau_2$ are two adjacent topes,  namely
$\overline \tau_1\cap  \overline \tau_2$ spans a hyperplane $H$.
The hyperplane $H$ is a rational subspace. Let $\tau_{12}$ be the
unique tope for $X\cap H$ such that $ \overline \tau_1\cap
\overline \tau_2 \subset \overline{\tau_{12}}$ (see Figure
\ref{tau12}).

\begin{example}
Let $C$ be the
 cone generated by the
vectors $a:=\omega_3+\omega_1$, $b:=\omega_3+\omega_2$,
$c:=\omega_3-\omega_1$, $d:=\omega_3-\omega_2$  in a
$3$-dimensional space $V:=\mathbb R \omega_1 \oplus \mathbb R
\omega_2\oplus \mathbb R \omega_3$. Figure \ref{tau12} represents
the section of $C$  cut by the   affine hyperplane containing
$a,b,c,d$. We consider $X:=[a,b,c,d]$.

We show  in  section, on the left of the picture,  two topes
$\tau_1,\tau_2$ adjacent along the hyperplane $H$ generated by
$b,d$ and, on the right, the tope $\tau_{12}$. The list $X\cap H$
is $[b,d]$. The closure of the tope $\tau_{12}$ is ``twice bigger ''
than $\overline \tau_1\cap \overline \tau_2$.

\end{example}

\begin{figure}[!h]
\begin{center}\includegraphics{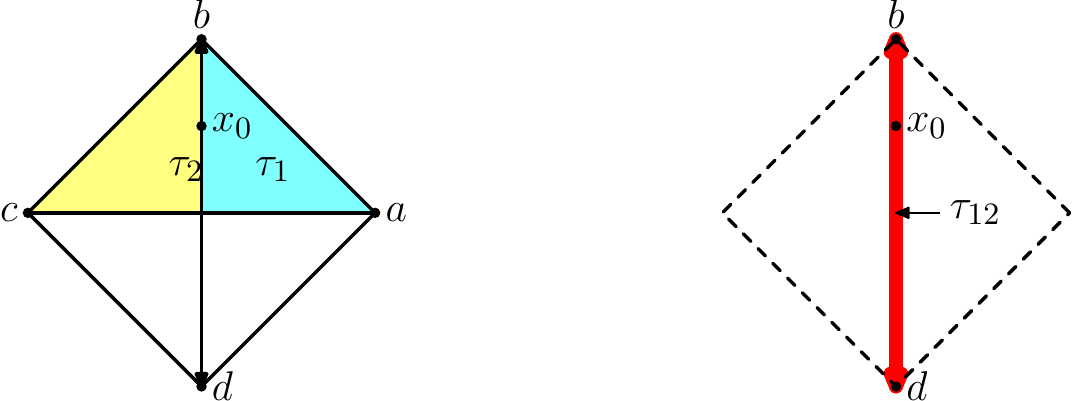}
\end{center}\caption{Two adjacent  topes of $X:=[a,b,c,d]$}\label{tau12}
\end{figure}

Let $f\in \mathcal F(X)$. The function $\nabla_{X\setminus H} f$
is an element of $\mathcal F(H\cap X)$, thus,  by Theorem
\ref{larcarsexplicit}, there exists a quasi--polynomial
$(\nabla_{X\setminus H}f)^{\tau_{12}}$ on $H$ such that
$\nabla_{X\setminus H}f$ agrees with $(\nabla_{X\setminus
H}f)^{\tau_{12}}$ on $\tau_{12}$.

\begin{theorem}\label{wcf}
Let $\tau_1,\tau_2,H,\tau_{12}$ be as before and
$f\in \mathcal F(X)$.  Let $F_H$ be the  half line  in  $H^{\perp}$ positive on $\tau_1$.
 Then
\begin{equation}\label{wallcrossing}
f^{\tau_1}-f^{\tau_2}=( \mathcal P_{X\setminus H}^{F_{H}} -\mathcal P_{X\setminus H}^{-F_{H} })*(\nabla_{X\setminus
H}f)^{\tau_{12}}.
\end{equation}
\end{theorem}

\begin{proof}

Let $x_0$ be a point   in the relative interior of $\overline
\tau_1\cap \overline \tau_2$ in  $H$.  Then  $x_0$  does not
belong to any $X$-rational hyperplane different from $H$ (see
Figure \ref{tau12}). Therefore  we can choose a regular vector
$u_{\underline r}$ for $X\setminus\underline r$ for every rational subspace $\underline r$ different from $H,V$
such that $u_{\underline r}$ is negative on $x_0$.  By continuity,
 there are points $x_1\in\tau_1$ and $x_2\in\tau_2$  sufficiently close to $x_0$ and where these elements  $u_{\underline r}$ are still negative.
We choose $u\in H^{\perp}$ positive on $\tau_1$.
Consider the sequences  ${\bf u^1}=(u^1_{\underline r} )$
where $u^1_{\underline r} =u_{\underline r} $ for ${\underline r}
\neq H$ and $u^1_H=-u$ and     ${\bf
u^2}=(u^2_{\underline r} )$ where $u^2_{\underline r}
=u_{\underline r} $ for ${\underline r} \neq H$ and $u^2_H=u$. Correspondingly we get two $X$-regular collections ${\bf F^1}$ and ${\bf F^2}$.

For $i=1,2$ let $f=f^i_V+ f^i_H+\sum_{\underline r\neq H,V}
f^i_{\underline r}$ be  the ${\bf F^i}$ de\-com\-po\-si\-tion of
$f$.

We write $f^1_H=\mathcal P_{X\setminus H}^{-F_{H}}*q^{(1)}$ with $q^{(1)}\in
DM(X\cap H)$. Now the sequence  ${\bf u^1}$ takes a negative value at
the point $x_1$ of $\tau_1$, thus  by
Theorem \ref{larcarsexplicit}, the component $f^1_V$ is equal
to $f^{\tau^1}$. By Proposition \ref{recurrence},
$$\nabla_{X\setminus H}
f=q^{(1)}+\sum_{\underline r\subset H, \underline r\neq H}
\nabla_{X\setminus H} f^1_{\underline r}$$ is the ${\bf F^1}$
de\-com\-po\-si\-tion of $\nabla_{X\setminus H} f$ so that, again  by Theorem
\ref{larcarsexplicit},
$q^{(1)}=(\nabla_{X\setminus H}f)^{\tau_{12}}$. We thus have $f_V^1=f^{\tau_1}$  and
$f^1_H={\mathcal P}_{X\setminus H}^{-F_{H} }*(\nabla_{X\setminus H}f)^{\tau_{12}}$.

Similarly  ${\bf u^2}$ takes a negative value at
the point $x_2$ of $\tau_2$, so $f^2_V=f^{\tau_2}$  and $f^2_H={\mathcal P}_{X\setminus H}^{F_{H}
}*(\nabla_{X\setminus H}f)^{\tau_{12}}$.

Now from Proposition \ref{gliop},   when $\dim(\underline r)=i$,
$$f^1_{\underline r}= \Pi_{X}^{\underline r,F^1_{\underline r}}{{\bf \Pi}_i^{\bf F^1}}f,\hspace{1cm}
f^2_{\underline r}= \Pi_{X}^{\underline r,F^2_{\underline r}}{{\bf \Pi}_i^{\bf F^2}}f,$$
and,  for  any $i<\dim V$, the operators   ${{\bf \Pi}_i^{\bf F^1}}$ and   ${{\bf \Pi}_i^{\bf F^2}}$ are equal.
Thus $f^1_{\underline r}=f^2_{\underline r}$ for $\underline
r\neq V,H$. So we obtain $f^1_V+f^1_H=f^2_V+f^2_H$, and  our
formula.

\end{proof}

Consider now the case where $X$ spans a pointed cone. Let us
interpret Formula (\ref{wallcrossing}) in the case in which
$f=\mathcal P _X$. We know that for a given tope $\tau$, $\mathcal
P _X$ agrees with a quasi--polynomial $\mathcal P _X^\tau$ on
$\tau-B(X)$. Recall that $\nabla_{X\setminus H}(\mathcal P
_X)=\mathcal P _{X\cap H}$ as we have seen in Lemma \ref{lacomin}.
It  follows that given two adjacent topes $\tau_1,\ \tau_2$ as
above, $(\nabla_{X\setminus H} f)^{\tau_{12}}$ equals $(\mathcal P
_{X\cap H})^{\tau_{12}}$ (extended by zero outside $H$). So we
deduce the identity
\begin{equation}\label{lidemer}
\mathcal P _X^{\tau_1}-\mathcal P _X^{\tau_2}= ({\mathcal P }_{X\setminus H}^{F_{H}}
-{\mathcal P }_{X\setminus H}^{-F_{H}})*\mathcal P _{X\cap H}^{\tau_{12}}.
\end{equation}
This is Paradan's  formula  (\cite{P1}, Theorem 5.2).

\begin{example} Assume $X=[a,b,c]$ as in Remark \ref{abc}.
We write $v\in V$ as $v=v_1 \omega_1+v_2\omega_2$.
 Let $\tau_1=\{
v_1>{v_2}>0\}$, $\tau_2=\{ 0<v_1<v_2\}$. Then one easily (see
Figure \ref{uu}) sees that
$$\mathcal P _X^{\tau_1}=({n_2}+1),
 \hspace{1cm} \mathcal P _X^{\tau_2}=(n_1+1),
 \hspace{1cm} \mathcal P _{X\cap H}^{\tau_{12}}=1.$$
 Equality (\ref{lidemer}) is equivalent to the following identity of
 series which is easily checked:
$$\sum_{n_1,{n_2}}({n_2}-n_1)x_1^{n_1}x_2^{n_2}
=(-\sum_{n_1\geq 0,{n_2}<0}x_1^{n_1}x_2^{n_2}+\sum_{n_1<
0,{n_2}\geq 0}x_1^{n_1}x_1^{n_2})(\sum_{h}x_1^hx_2^h).$$
\end{example}

Recall that  a  big cell  is a connected component of the
complement in $V$ of the {\it singular vectors},  which are formed
by the union of all cones $C(Y)$  for all the sublists $Y$ of $X$
which do not span $V$.  A big cell is usually larger than a tope.
See Figure \ref{bigcell} which shows a section of a cone in
dimension $3$ generated by $3$ independent vectors $a,b,c$. Here
$X=[a,b,c, a+b+c]$. On the drawing, the vertices $a$,$b$,$c$,$d$
represents the intersection of the section with the half lines
$\mathbb R^+ a$, $\mathbb R^+ b$, $\mathbb R^+ c$, $\mathbb R^+
d$.

\begin{figure}[!h]
\begin{center}\includegraphics{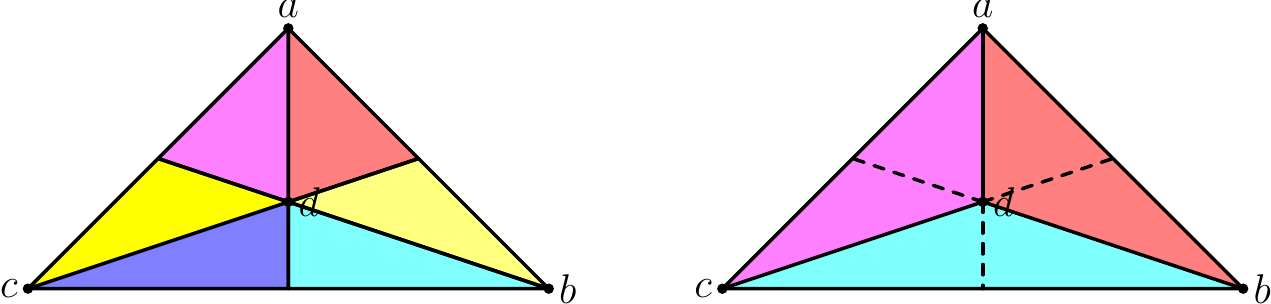}
\end{center}\caption{Topes and cells inside  $C(X)$ for  $X:=[a,b,c,d:=a+b+c]$}\label{bigcell}
\end{figure}

Let us now consider a big cell $\mathfrak c$. We need

\begin{lemma}\label{copert}
Given a big cell $\mathfrak c$,
  let $\tau_1,\dots,\tau_k$ be all the topes contained in $\mathfrak c$.  Then:
$$\mathfrak c-B(X)=\cup_{i=1}^k(\tau_i-B(X)  ).$$
\end{lemma}
\begin{proof}
Notice that $\cup_{i=1}^k\tau_i$ is dense in $\mathfrak c$. Given
$v\in \mathfrak c-B(X)$, $v+B(X)$ has non empty interior and thus
its non empty intersection with the open set $\mathfrak c$ has non
empty interior. It follows that $v+B(X)$ meets
$\cup_{i=1}^k\tau_i$  proving our claim.
\end{proof}

Now in order to prove the statement for big cells,
 we need to see
what happens when we cross a wall  between two adjacent topes.

\begin{theorem}\label{tbcp}
On  $\mathfrak c-B(X) $,
 the  partition function $\mathcal P _X$
agrees with a quasi--polynomial $f^{\mathfrak c}\in DM(X)$.
\end{theorem}
\begin{proof} By Lemma \ref{copert},
it  suffices to show that given two adjacent topes $\tau_1,\
\tau_2$ in $\mathfrak c$, $\mathcal P _X^{\tau_1}=\mathcal P
_X^{\tau_2}$.

But now notice that the positive cone spanned by $X\cap H$,
support of $\mathcal P _{X\cap H}$,  is formed of singular
vectors and therefore it is disjoint from $\mathfrak c$ by
definition of big cells.  Therefore $\mathcal P _{X\cap H}$
vanishes on $\tau_{12}$. Thus  $\mathcal P _{X\cap
H}^{\tau_{12}}=0$ and our claim follows from Formula
(\ref{lidemer}).
\end{proof}

This theorem was proven \cite{DM} by Dahmen-Micchelli for topes,
and by Szenes-Vergne \cite{SV1} for cells. In many cases, the sets
 $\mathfrak c-B(X) $ are the maximal domains of quasi--polynomiality for $\mathcal
 P_X$.

\begin{remark}
If  $\mathfrak c$ is a big cell contained in the cone $C(X)$, the open set
$\mathfrak c-B(X)$ contains $\overline{\mathfrak c}$ so that the
quasi--polynomial $f^{\mathfrak c}$ coincides with $\mathcal P _X$
on $\overline{\mathfrak c}$.

This is usually not so for $f\in \mathcal F(X)$ and a tope $\tau$ : the function $f$
does not usually coincide with  $f^{\tau}$ on $\overline{\tau}$.
Figure  \ref{uuuu} describes the  function  $f:=-\mathcal P_X^{F_{\underline r}}$
in $\mathcal F(X)$ with $X:=[a,b,c]$ as in Remark \ref{abc} and
$u\in F_{\underline r}$ strictly negative on $a,b$. We see for example that $f$ is
equal to $0$ on the set $n_1=0, n_2\leq 0$ which is in the closure
of the tope $\tau_3:=\{v_2<v_1<0\}$ while the quasi--polynomial
$f^{\tau_3}=(n_1+1)$ takes the value $-1$ there.
\begin{figure}
\begin{center}\includegraphics{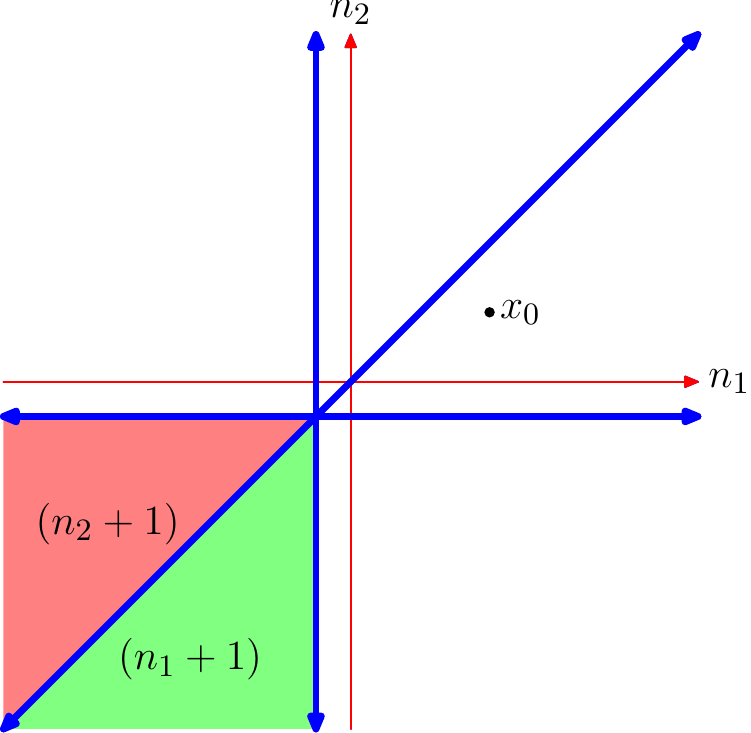}
\end{center}
\caption{ The  function $f$}\label{uuuu}
\end{figure}

\end{remark}

We finally give a  formula for  $\mathcal P_X$ due to Paradan.

Let us choose a scalar product on $V$.
We use the notations of Theorem \ref{before}.
As $\nabla _{X\setminus \underline r}\mathcal P_X=\mathcal P_{X\cap \underline r}$, we obtain as a corollary of
Theorem \ref{before}

\begin{theorem} \label{beautiful}(Paradan).
 Let  $\beta\in V$  be
 generic   with respect to all the rational subspaces $\underline r$. Then, we have
$$\mathcal P_X=\sum_{\underline r\in  S_X} {\mathcal
P}_{X\setminus\underline r}^{-F_{\underline r }^\beta}*(\mathcal P_{X\cap \underline r})^{\tau(p_{\underline r}\beta)}.$$
\end{theorem}

\begin{remark}
The set of  $\beta\in V$
 generic   with respect to all the rational subspaces $\underline r$ decomposes into finitely many open polyhedral cones.

 The decomposition depends only on the cone in which $\beta$ lies.

In this decomposition, the component in $DM(X)$ is the
 quasi-polynomial which coincides with $\mathcal P_X$ on the cell containing $\beta$.

 Finally if  $X$ spans a pointed cone and $\beta $ has negative scalar product with $X$,  the decomposition reduces to  $\mathcal P_X= \mathcal P_X $, the component for $\underline r=\{0\}$. \end{remark}

\section{A second remarkable space}
\subsection{A de\-com\-po\-si\-tion formula}

In this section,  we want to present the analogue for distributions,
the proofs are essentially the same or simpler than in the
previous case, so we skip them.
 We shall freely use the notations of the previous sections.

Let $V$ be a finite dimensional vector space, consider the space
$\mathcal D(V)$ of  distributions  on $V$. We denote by $\delta_0$
the delta distribution on $V$.

$\mathcal D(V)$ is in an obvious way
a module over  the algebra  of distributions with compact support,
 under convolution.
Let now $X:=[a_1,\dots,a_m]$ be a list of non--zero elements of $V$.

\begin{definition}\begin{enumerate}
\item Given a rational subspace $\underline r$,
  we denote  by
$\mathcal D(V, \underline r)$  the set of elements in $\mathcal D(V)$ which vanish
on all test functions vanishing  on  $  \underline r$.

\item Given a vector $a\neq 0$,
 we denote by $\partial_a$ the directional
derivative associated to $a$. For a list $Y$ of non--zero vectors,
we denote   by $\partial_{Y}:=\prod_{a\in Y}\partial_a$.
\end{enumerate}\end{definition}
The restriction map $C_c^{\infty}(V)\to
C_c^{\infty}(\underline r)$ on test functions induces, by duality, an  identification  between the space of  distributions on $\underline r$
and  the space $\mathcal D(V,\underline r)$.

If  $X$ spans $V$,  we consider the space  defined by
Dahmen--Micchelli,  which we denote $D(X)$,   formed
by the  distributions $f\in \mathcal D(V)$  satisfying the system of
differential equations $\partial_Yf=0$,  as $Y$ varies among all cocircuits
of $X$. It is easy to see that an element of $D(X)$ is a
polynomial density $P(x)dx$ on $V$.

\medskip

Assume that $X$ spans a pointed cone. Recall that the
{\it multivariate spline $T_X$}   is the tempered distribution
defined, on test functions $f$, by:
\begin{equation}\label{multiva}
\langle T_X\,|\,f\rangle = \int_0^\infty\cdots\int_0^\infty
f(\sum_{i=1}^mt_i a_i)dt_1\cdots dt_m.
\end{equation}

If   $W$ is the span of $X$ and if we choose a Lebesgue measure
$dx$ on $W$,
    we may interpret  $T_X$ as a function on
$W$ supported in the cone $C(X)$ by writing  $\langle
T_X\,|\,f\rangle =\int_{W}f(x)T_X(x)dx.$

If  $Y$ is a sublist of $X$,   one has that $\partial_{X\setminus
Y}T_X=T_Y$.\smallskip

 We next define
 the vector space of interest for this section:
\begin{equation}
\label{ilpp}  G(X):=
\{f\in \mathcal D(V)\,|\,
\partial_{X\setminus\underline r}f\in \mathcal D(V, \underline r),
\text{ for all } \underline r\in S_X\}.
\end{equation}

\begin{lemma}\label{cont}
\begin{enumerate}
\item If $X$ generates a pointed cone, the  multivariate spline
$T_X $ lies in $G(X)$. \item The space $D(X)$ is contained in
$G(X)$.
\end{enumerate}\end{lemma}
\begin{proof}
{\it i)}  $\partial_{X\setminus\underline r}T_X= T_{X\cap\underline r}\in
\mathcal D(V, \underline r).$

{\it ii)} is clear from the definition.

\end{proof}

As for the partition functions,
    this lemma is  a very special
case   of Theorem \ref{gestad1} which follows.

Given a rational subspace $\underline r$, take  a regular face $F_{\underline r}$ for $X\setminus \underline r$.
 Divide as before the set  $X\setminus \underline r$ into
two parts $A,B$, of positive and negative vectors on $F_{\underline r}$

We want to define an element $T_{X\setminus \underline r}^{F_{\underline r}}\in \mathcal D(V)$ which is
characterized by the following two properties:
\begin{lemma} There exists a unique element  $T_{X\setminus r}^{F_{\underline r}}$
characterized by the properties $\partial_{X\setminus\underline r} T_{X\setminus r}^{F_{\underline r}}=
\delta_0 $ and $T_{X\setminus r}^{F_{\underline r}}$  is supported in    $ C(A,-B)$.
\end{lemma}
\begin{proof}  Set:
$$T_{X\setminus r}^{F_{\underline r}}=(-1)^{|B|}T_{[A, -B]} .$$
It  is easily seen that this element satisfies the two properties.
The uniqueness is also clear.
\end{proof}
Identify  the space of Dahmen--Micchelli   $D(X\cap \underline r)$
with a subspace of  $\mathcal D(V, \underline r)$. Although a distribution
$f\in \mathcal D(V, \underline r)$ may have  non compact support,
 we easily see that the convolution product
$T_{X\setminus r}^{F_{\underline r}}*f$ is well defined. In fact, given any $\gamma\in V$, we
can write $\gamma=\lambda+\mu$ with $\mu\in \underline r, $ and
$\lambda\in C(A,-B)$   only in  a bounded polytope.

The analog of the ``mother formula'' (\ref{eqmother}) of Proposition
\ref{promother} is the following formula.

For $g\in \mathcal D(V,\underline r)$:
\begin{equation}\label{eqmothercontinuous}
\partial_{X\setminus \underline t} (T_{X\setminus r}^{F_{\underline r}}*g) ={T}_{(X\setminus\underline r)\cap \underline t}^{F_{\underline r}}*
(\partial_{(X\cap \underline r)\setminus (\underline t \cap \underline
r)}g).
\end{equation}

Following the same scheme of proof as for Theorem \ref{gesta}, the
following theorem follows:
\begin{theorem}\label{gestad1}
 Choose for every rational space $\underline r$,
  a  regular face $F_{\underline r}$ for  $X\setminus \underline r$. Then:
$$G(X)=\oplus_{\underline r\in S_X}
T_{X\setminus r}^{F_{\underline r}}*D(X\cap \underline r ).$$
\end{theorem}

We associate, to a rational space ${\underline r}$ and a  regular face $F_{\underline r}$ for  $X\setminus \underline r$,
 the operator on $G(X)$ defined by
$$\pi_X^{\underline r ,F_{\underline r}}: f\mapsto
T_{X\setminus r}^{F_{\underline r}}*(\partial_{X\setminus {\underline r} }f).$$ This is well defined.
Indeed $\partial_{X\setminus {\underline r} }f$ is supported on
${\underline r} $ so that the convolution is well defined. We see
that $\pi_X^{\underline r ,F_{\underline r}}$ maps $G(X)$ to $G(X)$ and that it
is a projector.

Given a $X$--regular collection ${\bf F}$, we can write an element
$f\in G(X)$ as
$$f=\sum_{\underline r\in S_X}f_{\underline
r}$$ with $f_{\underline r}\in T_{X\setminus r}^{F_{\underline r}}*D(X\cap
\underline r ).$  This expression for $f$ will be called  the
${\bf F}$ de\-com\-po\-si\-tion of $f$.
 In this
de\-com\-po\-si\-tion, the component $f_V$ is in $D(X).$

The space $T_{X\setminus r}^{F_{\underline r}}*D(X\cap \underline r )$
 will be referred to as {\it the $F_{\underline r}$-component} of $G(X)$.

\smallskip

Let ${\bf F}$  be a $X$--regular collection.  One can write in the same way as in Proposition \ref{gliop} the explicit projectors  to the  various components.
\subsection{Polynomials}

Let $f\in G(X)$ and let $\tau$ be a tope.
\begin{proposition}[Localization theorem]\label{larcarsexplicit1}
   Let ${\bf F}$ be a $X$-regular    collection non-positive on $\tau$.
 Let $f=\sum f_{\underline
r}$ be the ${\bf F}$-de\-com\-po\-si\-tion of $f$. Then the
component $f_V$ of this de\-com\-po\-si\-tion  is a  polynomial
density in $D(X)$ such that $f=f_V$ on $\tau$.
\end{proposition}
\begin{proof}
Write $f=\sum_{\underline r\in S_X}f_{\underline r}$ with
$f_{\underline r}=T_{X\setminus \underline r}^{F_{\underline r}}* k_{\underline r}$ where
$k_{\underline r}\in D(X\cap {\underline r})$. The distribution
$f_{\underline r}=T_{X\setminus r}^{F_{\underline r}}* k_{\underline r}$ is
supported on  ${\underline r} +C(F_{\underline r},X\setminus \underline r)$. As in
Theorem \ref{larcarsexplicit},  we know that $\tau\cap
({\underline r} +C(F_{\underline r} ,X\setminus \underline r))=\emptyset$.\end{proof}

\begin{remark}
Thus the distribution $f$ is a locally polynomial density on $V$. In particular this is a tempered distribution. Here the distribution  $f$ coincides with  the polynomial density
$f_V$ only on $\tau$ and not on the bigger open set $\tau-B(X)$.
This extension property is replaced the regularity property that
$f$ is of class $C^{r-1}$, where $r$ is the minimum of the
cardinality of the cocircuits of $X$ (see \cite{dp1}).

\end{remark}

We shall denote by $f^\tau$  the polynomial density in $D(X)$
coinciding with $f$ on the tope $\tau$.\smallskip

Let $H$ be a rational hyperplane,  let $u\in U$ be an equation
of the hyperplane and $F_H$ the half-line containing $u$.
\begin{lemma}
If $q\in D(X\cap H)$, then  $w:=(T_{X\setminus H}^{F_{H}}- T_{X\setminus H}^{-F_{H}})*q$ is an
element of $D(X)$.
\end{lemma}

Let us use  the notations  $\tau_1,\tau_2,H,\tau_{12}$  as in Theorem \ref{wcf}.
 Let $f\in G(X)$. The distribution $\partial_{X\setminus H} f$ is an element of
$G(H\cap X)$, thus by Proposition \ref{larcarsexplicit1}, there
exists a polynomial density $(\partial_{X\setminus H}f)^{\tau_{12}}\in
D(X\cap H)$ on $H$ such that $\partial_{X\setminus H}f$ agrees with
$(\partial_{X\setminus H}f)^{\tau_{12}}$ on $\tau_{12}$.

\begin{theorem}
Let $F_H$ be the half line in $H^{\perp}$ positive on $\tau_1$.  Then
\begin{equation}\label{wallcrossing1}
f^{\tau_1}-f^{\tau_2}= (T_{X\setminus H}^{F_{H}}- T_{X\setminus H}^{-F_{H}})*(\partial_{X\setminus H}f)^{\tau_{12}}.
\end{equation}
\end{theorem}

When  $X$ spans a pointed cone, we interpret Formula
(\ref{wallcrossing1}) for $f=T_X$. On a given tope $\tau$,  $T_X$
agrees with a  polynomial density  $T _X^\tau$.

Since $\partial_{X\setminus H}(T _X)=T_{X\cap H}$  we deduce, the
identity
\begin{equation}\label{lidemer1}
T _X^{\tau_1}-T _X^{\tau_2}= (T_{X\setminus H}^{F_{H}}- T_{X\setminus H}^{-F_{H}})*T _{X\cap
H}^{\tau_{12}}.
\end{equation}

Now the statement for big cells:

\begin{theorem}\label{polcell}
On  $\mathfrak c  $,
 the  multivariate spline $T _X$
agrees with a  polynomial density  in $D(X)$.
\end{theorem}

\begin{remark}
Once Theorems \ref{tbcp} and \ref{polcell} have been proven, it is
easy to deduce from them that the generalized  Khovanskii--Pukhlikov  formula relating volumes and number of points holds
\cite{BriVer97}. Indeed, one can prove it easily sufficiently far
from the walls (cf. \cite{dp1}).
\end{remark}

Finally, using a scalar product, we give Paradan's formula for $T_X$.
\begin{theorem} \label{beautifulbis}(Paradan).
 Let  $\beta\in V$  be
 generic   with respect to all the rational subspaces $\underline r$. Then, we have
$$T_X=\sum_{\underline r\in  S_X}  {T}_{X\setminus\underline r}^{-F_{\underline r }^\beta}*(T_{X\cap \underline r})^{\tau(p_{\underline r}\beta)}.$$
\end{theorem}

\end{document}